\documentclass[12pt]{amsart}

\usepackage{amsmath,amssymb,lscape}
\usepackage{graphicx}
\usepackage{epstopdf}

\newtheorem{theorem}{Theorem}[section]
\newtheorem{proposition}[theorem]{Proposition}
\newtheorem{lemma}[theorem]{Lemma}

\theoremstyle{definition}
\newtheorem{definition}[theorem]{Definition}
\newtheorem{example}[theorem]{Example}

\theoremstyle{remark}
\newtheorem{remark}[theorem]{Remark}

\numberwithin{equation}{section}

\def\bZ{{\mathbb {Z}}}

\def\bR{{\mathbb {R}}}

\def\pB{{\mathcal B}}
\def\pD{{\mathcal D}}
\def\pE{{\mathcal E}}
\def\pF{{\mathcal F}}

\def\pG{{\mathcal G}}

\def\pP{{\mathcal P}}

\def\bd{\frak d}
\def\be{\frak e}
\def\bf{\frak f}
\def\bg{\frak g}

\def\VB{{\mathcal {VB}}}
\def\VCD{{\mathcal {VCD}}}
\def\Aut{{\rm Aut}}

\begin{document}

\title{Virtual braids and virtual curve diagrams}
\author{Oleg Chterental}

\maketitle
\begin{center}\today\end{center}

\begin{abstract}
There is a well known injective homomorhpism $\phi:\pB_n \rightarrow \Aut(F_n)$ from the classical braid group $\pB_n$ into the automorphism group of the free group $F_n$, first described by Artin \cite{A}. This homomorphism induces an action of $\pB_n$ on $F_n$ that can be recovered by considering the braid group as the mapping class group of $H_n$ (an upper half plane with $n$ punctures) acting naturally on the fundamental group of $H_n$.

Kauffman introduced virtual links \cite{Ka} as an extension of the classical notion of a link in $\bR^3$. There is a corresponding notion of a virtual braid. In this paper, we will generalize the above action to virtual braids. We will define a set, $\VCD_n$, of ``virtual curve diagrams" and define an action of $\VB_n$ on $\VCD_n$. Then, we will show that, as in Artin's case, the action is faithful. This provides a combinatorial solution to the word problem in $\VB_n$. 

In the papers \cite{B,M}, an extension $\psi:\VB_n\rightarrow \Aut(F_{n+1})$ of the Artin homomorphism was introduced, and the question of its injectivity was raised. We find that $\psi$ is not injective by exhibiting a non-trivial virtual braid in the kernel when $n=4$.
\end{abstract}

\section{Introduction}

Virtual knot theory was introduced by Kauffman in \cite{Ka}. It is an extension of classical knot theory in the sense that the set of isotopy classes of classical links embeds naturally into the set of virtual isotopy classes of virtual links. Virtual links have a topological interpretation as links in thickened surfaces up to the addition/removal of trivial handles and diffeomorphisms of the thickened surface \cite{CKS}. Such a representation of minimal genus is unique \cite{Ku}. Many classical link invariants such as the fundamental group of the complement, the Alexander and Jones polynomials, and Khovanov homology, have multiple extensions to virtual links. For further information and references regarding virtual knot theory, the reader may consult \cite{MI}.

There is a relationship between virtual links and closed virtual braids, via virtual versions of the Alexander and Markov theorems \cite{K, KL}. The set of virtual braids on $n$ strands forms a group, $\VB_n$, that has been studied from several angles \cite{BMVW, C, GP}.

While we are primarily interested in the virtual braid group $\VB_n$, we start with a brief review of the ordinary braid group and its ``Artin representation". Our work on $\VB_n$ is then a generalization of the classical case.

\subsection{Review of braids and the Artin representation}

Recall the braid group $\pB_n$ on $n$ strands. The standard presentation of $\pB_n$ has generators $\sigma_1, \ldots, \sigma_{n-1}$ and relations $\sigma_i \sigma_j = \sigma_j \sigma_i$ for $|i-j|>1$ and $\sigma_i \sigma_{i+1} \sigma_i=\sigma_{i+1} \sigma_i \sigma_{i+1}$ for $1 \leq i \leq n-2$. Braid words can be represented by braid diagrams. Each generator has a braid diagram as in Figure \ref{bdiag}. The diagram of a product $\alpha \beta$ is obtained by stacking the diagram of $\alpha$ on top of the diagram for $\beta$.

\begin{figure}
\begin{center}
\includegraphics[scale=1.5]{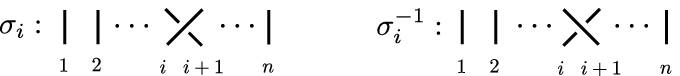}
\caption{Braid diagrams for the braid group generators.}\label{bdiag}
\end{center}
\end{figure}

Let $F_n$ be the free group on $n$ symbols $x_1,x_2, \ldots, x_n$. Artin \cite{A} studied a homomorphism $\phi$\index{Artin action} from $\pB_n$ to the automorphism group $\Aut(F_n)$ defined by \[ \phi(\sigma_i)(x_j)=\begin{cases}x_j & j \neq i,i+1 \\ x_ix_{i+1}x_i^{-1} & j=i \\ x_i & j=i+1 \end{cases}. \] We will alternatively refer to this homorphism as the \textbf{Artin action} or the Artin representation. If $\beta$ is a braid on $n$ strands then there is a permutation $\pi$ of $\{1,2,\ldots,n\}$ and an element $U_j \in F_n$ for each $1 \leq j \leq n$ such that $\phi(\beta)(x_j)=U_jx_{\pi(j)}U_j^{-1}$. The elements $U_j$ are uniquely determined up to multiplication on the right by $x_{\pi(j)}^{\pm 1}$. Thus the vector of free group cosets \[ (U_1 \langle x_{\pi(1)} \rangle, U_2 \langle x_{\pi(2)} \rangle, \ldots, U_n \langle x_{\pi(n)} \rangle), \] completely determines the map $\phi(\beta)$ on the free group. The cosets (and by abuse of terminology, the $U_j$'s) will be referred to as the \textbf{braid coordinates}\index{braid coordinates} of $\beta$.

\begin{figure}
\begin{center}
\includegraphics[scale=1.5]{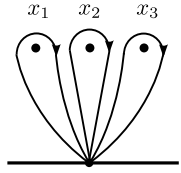}
\caption{The generators of the fundamental group of $H_3$.}\label{samp}
\end{center}
\end{figure}

\begin{figure}
\begin{center}
\includegraphics[scale=1]{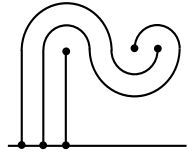}
\caption{The curve diagram of $\sigma_2^{-1} \sigma_1 \sigma_2 \in \pB_3$.}\label{exdiag}
\end{center}
\end{figure}

Let $H_n$ be the upper half-plane $\bR \times [0,\infty)$ with $n$ punctures lying on some horizontal line above the boundary. Choose a base point on the boundary. Consider a closed path starting at the base point and winding once clockwise around the $j^{\rm th}$ puncture from the left as in Figure \ref{samp}. By associating this path to the free group generator $x_j$, the fundamental group of $H_n$ can be identified with $F_n$. The coordinates of a braid can then be depicted as a collection of $n$ curves, each curve connecting the base point to one of the punctures in $H_n$.

The above approach to the Artin action has been used in \cite{Br} to give a normal form for braids, in \cite{DDRW} to give another derivation of the Dehornoy ordering of braids, and in \cite{DW} to give another polynomial time solution to the word problem in the braid group. It suffices to say that the curve diagram approach provides an interesting perspective on the braid group.

For example, Figure \ref{exdiag} shows the curve diagram of the braid $\sigma_2^{-1} \sigma_1 \sigma_2 \in \pB_3$. For the sake of neatness, we have $n$ different base points ordered from left to right. The curve starting at the $i^{\rm th}$ base point from the left represents the $i^{\rm th}$ braid coordinate. The curves are taken up to homotopy relative to the base points and puncture points. One verifies that the diagram in Figure \ref{exdiag} is correct by calculating the action of $\phi(\sigma_2^{-1} \sigma_1 \sigma_2)$ on the free group generators $x_1$, $x_2$ and $x_3$. The braid coordinates of $\sigma_2^{-1} \sigma_1 \sigma_2$ are \[ (x_1\langle x_3 \rangle, x_1x_3^{-1}\langle x_2\rangle ,\langle x_1\rangle),\] and this agrees with the diagram.

\begin{figure}
\begin{center}
\includegraphics[scale=0.85]{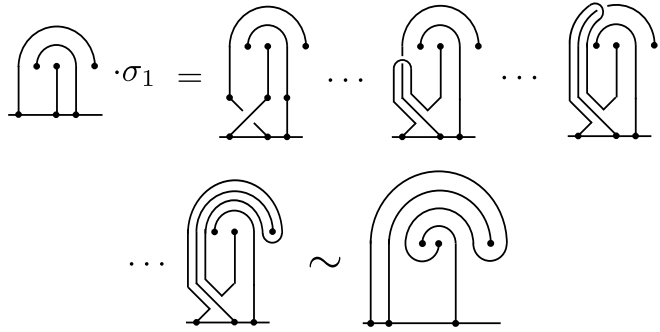}
\caption{The right action of $\pB_n$ on curve diagrams.}\label{rightpushaction}
\end{center}
\end{figure}

If the braid coordinates of $\beta$ are known, then one can determine the braid coordinates of $\sigma_i^{\pm 1} \beta$ and $\beta \sigma_i^{\pm 1}$. This induces both a left action and a right action of $\pB_n$ on curve diagrams with $n$ curves. Both actions can be described by a ``pushing off" procedure, for example as in Gaifullin-Manturov \cite{GM}. If $D$ is a curve diagram then the curve diagram $D \cdot \sigma_i^{\pm 1}$ is obtained by stacking $D$ on top of the braid diagram for $\sigma_i^{\pm 1}$ and pushing the upper strand off of the lower strand. This is shown in Figure \ref{rightpushaction}. The curve diagram $\sigma_i^{\pm 1} \cdot D$ is obtained by stacking the curve diagram for $\sigma_i^{\pm 1}$ on top of the curve diagram $D$ and pushing off the curves. The stacking is done in such a way that the bottom endpoints of the diagram of $\sigma_i^{\pm 1}$ are extended below the curves of $D$ to connect to the puncture points of $D$. This is shown in Figure \ref{leftpushaction}.

\begin{figure}
\begin{center}
\includegraphics[scale=0.65]{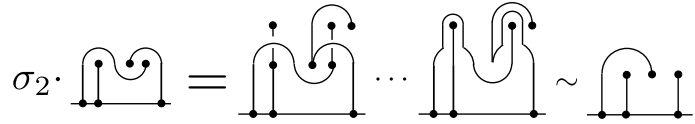}
\caption{The left action of $\pB_n$ on curve diagrams.}\label{leftpushaction}
\end{center}
\end{figure}

In Figure \ref{pushing}, the pushing off procedure is used to compute the right action of the braid word $\sigma_2^{-1} \sigma_1 \sigma_2$ on $I_3$, the trivial curve diagram with $3$ curves. The procedure returns the curve diagram in Figure \ref{exdiag}. This is expected, since for any braid $\beta$ we have $\beta \cdot I_n=I_n \cdot \beta$ and both are equal to the curve diagram for that braid. Artin proved that $\phi$ is injective, so a braid is determined by its braid coordinates or equivalently by its curve diagram.

\begin{figure}
\begin{center}
\includegraphics[scale=0.65]{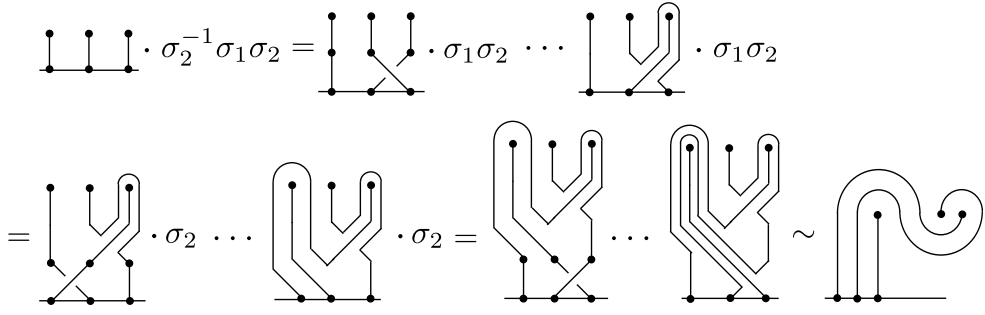}
\caption{The right action of $\sigma_2^{-1} \sigma_1 \sigma_2 \in pB_3$ on $I_3$, the trivial curve diagram with three curves.}\label{pushing}
\end{center}
\end{figure}

\subsection{Motivating virtual curve diagrams for virtual braids}

All of this can be repeated, with some modifications, in the setting of virtual braids. Let $\VB_n$ be the virtual braid group on $n$ strands. This group has a presentation with generators $\sigma_1, \ldots, \sigma_{n-1}$ and $\tau_1, \ldots, \tau_{n-1}$. The generators $\sigma_1, \ldots, \sigma_{n-1}$ satisfy the ordinary braid group relations. The generators $\tau_1, \ldots, \tau_{n-1}$ generate the symmetric group on $n$ symbols and satisfy the relations $\tau_i \tau_j = \tau_j \tau_i$ for $|i-j|>1$ and $\tau_i \tau_{i+1} \tau_i=\tau_{i+1} \tau_i \tau_{i+1}$ for $1 \leq i \leq n-2$ and $\tau_i^2=1$ for $1 \leq i \leq n-1$. Additionally, there are mixed relations $\sigma_i \tau_j = \tau_j \sigma_i$ for $|i-j|>1$ and $\tau_{i+1} \sigma_i \tau_{i+1} = \tau_i \sigma_{i+1} \tau_i$ for $1 \leq i \leq n-2$. Virtual braid words can be represented by virtual braid diagrams. The $\tau_i$ generator has a virtual braid diagram as in Figure \ref{vbdiag}. As in the braid case, the diagram of a product $\alpha \beta$ is obtained by stacking the diagram of $\alpha$ on top of the diagram for $\beta$.

\begin{figure}
\begin{center}
\includegraphics[scale=1]{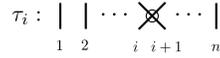}
\caption{The virtual braid diagram for the virtual crossing $\tau_i$.}\label{vbdiag}
\end{center}
\end{figure}

Before describing the action we will ultimately adopt (see Section \ref{vcsec}), we start with a naive extension of the Artin representation to the virtual case. Define a homomorphism $\psi:\VB_n \rightarrow \Aut(F_n)$ by the formula $\psi(\sigma_i)=\phi(\sigma_i)$ and \[ \psi(\tau_i)(x_j)=\begin{cases}x_j & j \neq i,i+1 \\ x_{i+1} & j=i \\ x_i & j=i+1 \end{cases}. \] For any virtual braid $\beta$, there is a permutation $\pi$ of $\{1,2,\ldots,n\}$ and an element $U_j \in F_n$ for each $1 \leq j \leq n$ such that $\psi(\beta)(x_j)=U_jx_{\pi(j)}U_j^{-1}$. The elements $U_j$ are uniquely determined up to multiplication on the right by $x_{\pi(j)}^{\pm 1}$. To any virtual braid $\beta \in \VB_n$ we assign the $n$-tuple of free group cosets \[ (U_1 \langle x_{\pi(1)} \rangle, U_2 \langle x_{\pi(2)} \rangle, \ldots, U_n \langle x_{\pi(n)} \rangle), \] which we will again refer to as the coordinates of $\beta$. The coordinates of a virtual braid may be represented as curves in a punctured half plane in the same way as the coordinates of an ordinary braid. The curves are taken up to homotopy fixing the base points and the punctures. For example, the coordinates of the two virtual braids $\tau_2 \sigma_1 \sigma_2$ and $\sigma_1\sigma_2\tau_1$ in $\VB_3$ are \[(x_1 \langle x_3 \rangle,x_1 \langle x_2 \rangle,\langle x_1\rangle ), \] and the corresponding curve diagram is shown in Figure \ref{vex} (up to a homotopy of the curves). As we will see, these two virtual braids are distinct and $\psi$ is not injective. In fact, it is known that the kernel of $\psi$ is generated by $\tau_2 \sigma_1 \sigma_2 (\sigma_1\sigma_2\tau_1)^{-1}$ (in the $n=3$ case). For larger $n$, it is generated by all braids obtained by adding trivial strands to the left and right of this braid. The quotient $\VB_n / \ker \psi$ is known by several names: the welded braid group, the group of basis-conjugating autmorphisms of $F_n$, and the loop braid group. See \cite{BD} for more information and references regarding welded braids.

In \cite{B,M}, another more intricate extension $\psi:\VB_n \rightarrow \Aut(F_{n+1})$ of the Artin representation was given, mapping a virtual braid on $n$ strands to an automorphism of the free group on $n+1$ generators $x_1,x_2,\ldots,x_n,q$. Up till now, the injectivity of this $\psi$ has been open. However, we will show in Section \ref{fut} that this representation has a non-trivial kernel as well.

\begin{figure}
\begin{center}
\includegraphics[scale=1]{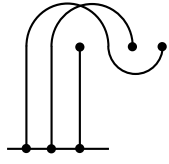}
\caption{The curve diagram of $\tau_2 \sigma_1 \sigma_2$ and $\sigma_1\sigma_2\tau_1$ in $\VB_3$.}\label{vex}
\end{center}
\end{figure}

In search of a virtual braid invariant that is stronger than $\psi$, we alter the definition of a curve diagram. Recall that the $n$ punctures of $H_n$ lie on a horizontal line above the boundary. We call this line the upper line. The different ways that a homotopy of the curves in a curve diagram can interact with the upper line are shown in Figure \ref{homotopies}. By restricting the allowed homotopies to exclude the homotopy labelled $F$, one gets a finer equivalence relation on curve diagrams. The $F$ stands for forbidden\index{forbidden $F$-move}, in reference to the forbidden moves of virtual knot theory. These finer equivalence classes of curve diagrams will be called \textbf{virtual curve diagrams}\index{virtual curve diagram} (see Section \ref{vcsec} for a more detailed definition). We denote the set of virtual curve diagrams with $n$ curves by $\VCD_n$.

\begin{figure}
\begin{center}
\includegraphics[scale=0.65]{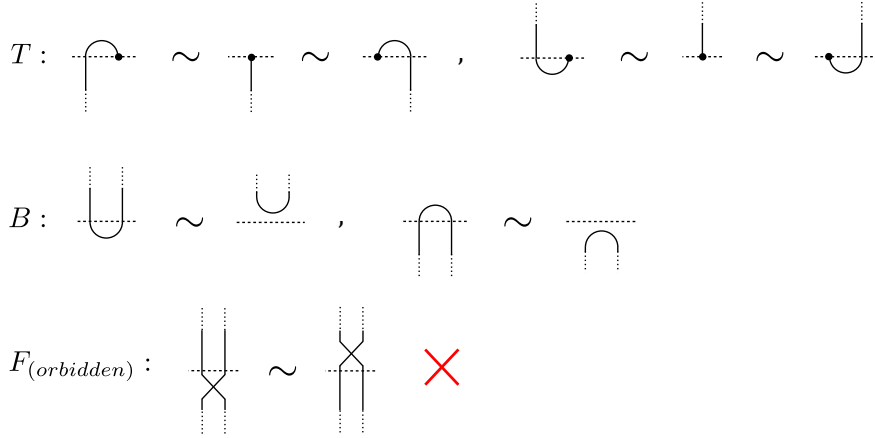}
\caption{Homotopies of curve diagrams interacting with the upper line.}\label{homotopies}
\end{center}
\end{figure}

There are well-defined left and right actions of $\VB_n$ on $\VCD_n$ via the pushing off procedure. Returning to the example of $\tau_2 \sigma_1 \sigma_2$ and $\sigma_1\sigma_2\tau_1$ in $\VB_3$, Figure \ref{wmove} shows that the right action of both braids on $I_3$ yields distinct virtual curve diagrams, thanks to the absence of the forbidden homotopy $F$.

\begin{figure}
\begin{center}
\includegraphics[scale=0.75]{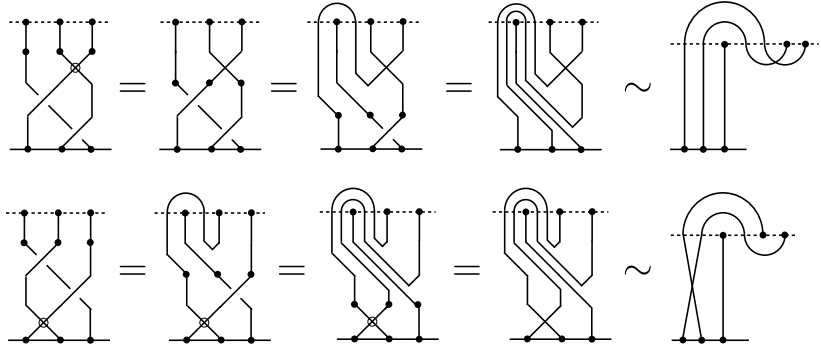}
\caption{The right action of $\tau_2 \sigma_1 \sigma_2,\sigma_1\sigma_2\tau_1 \in \VB_3$ on $I_3$, yields distinct virtual curve diagrams.}\label{wmove}
\end{center}
\end{figure}

We will show that the map from $\VB_n$ to $\VCD_n$ given by $\beta \mapsto \beta \cdot  I_n$ is injective ($I_n$ is the trivial virtual curve diagram with $n$ curves). Virtual curve diagrams are easy to tell apart, hence we solve the word problem for $\VB_n$. Given a braid, the corresponding curve diagram can be calculated in polynomial time in the length of the braid. This fact is briefly explained in \cite{DW} (though in that reference, the real focus is in the opposite direction: extracting a braid word from a curve diagram in polynomial time). It seems likely that in a similar manner virtual curve diagrams can be calculated quickly.

We note that the word problem for $\VB_n$ has already been solved in \cite{GP}. The methods there apply to a general class of Artin-Tits groups that includes $\VB_n$, and additionally it is shown that $\VB_n$ has virtual cohomological dimension $n-1$. There is as well a geometric/topological approach to virtual braids initiated in \cite{C}.

In Section \ref{vcsec} we introduce virtual curve diagrams, give an example, and define equivalence of virtual curves diagrams. In Section \ref{vaction} we describe the left action of $\VB_n$ on $\VCD_n$ and show that it is well-defined. In Section \ref{inj} we prove that the left action of $\VB_n$ on $\VCD_n$ is faithful. In Section \ref{fut} we describe an element of the kernel of the map $\psi$ of Bardakov and Manturov.

\section*{Acknowledgements}
I'd like to thank my Ph.D. advisor Dror Bar-Natan for his support during the course of my degree and feedback during the writing of this paper.

I also wish to thank Valeriy G. Bardakov, Paolo Bellingeri, Hans U. Boden, Bruno A. Cisneros De La Cruz, Louis H. Kauffman, Vassily Manturov and Emmanuel Wagner for several helpful discussions and comments.

\section{Virtual curve diagrams}\label{vcsec}

\subsection{A formal definition and an example}

For the following definition, it may be helpful to refer to Example \ref{ex1} along with Figure \ref{exconv}.

\begin{figure}
\begin{center}
\includegraphics[scale=1.5]{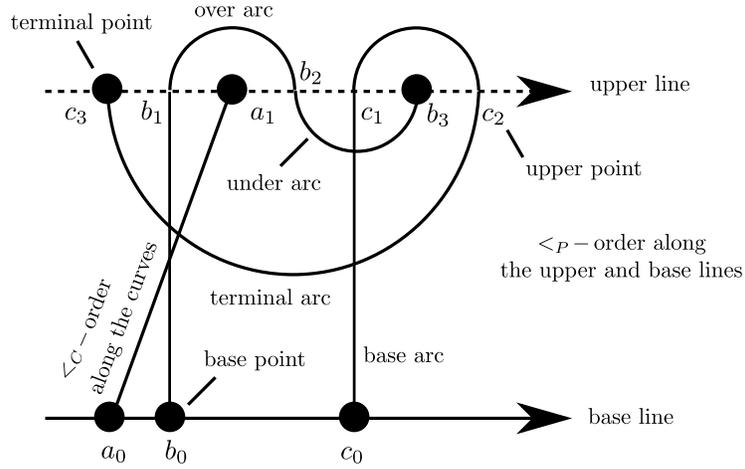}
\caption{A vcd with $3$ curves. See Definition \ref{vcd} and Example \ref{ex1}. The diagram consists of $7$ upper points, $3$ base points, $3$ terminal points, $3$ base arcs, $2$ over arcs, $2$ under arcs and $3$ terminal arcs.}\label{exconv}
\end{center}
\end{figure}

\begin{definition}\label{vcd} A \textbf{virtual curve diagram}\index{virtual curve diagram} (or vcd) $D$ with $n$ curves is a tuple $(P,<_P,<_C)$ where:
\begin{itemize}
\item $P$ is a finite set whose elements are called \textbf{points}\index{point}.
\item The pair $(P,<_P)$\index{$<_P$-order} is a partially ordered set isomorphic to two disjoint chains $U$ and $B$. The elements of $U$ are \textbf{upper points}\index{point!upper} and the elements of $B$ are \textbf{base points}\index{point!base}.
\item The pair $(P, <_C)$\index{$<_C$-order} is a partially ordered set that is isomorphic to $n=|B|$ disjoint chains called \textbf{curves}\index{curve}. Every curve's minimal point is a base point and every curve ends at an upper point, the \textbf{terminal point}\index{point!terminal}. Pairs of points $(a,b)$ where $b$ $<_C$-covers $a$ will be called \textbf{arcs}\index{arc}. Note that since terminal points are not base points, every curve contains at least two points.
\item Assume some curve consists of the points $a_0 <_C a_1 <_C \ldots <_C a_r$, where $a_0$ is a base point and $a_r$ is a terminal point. The arc $(a_0,a_1)$ is the \textbf{base arc}\index{arc!base}. For $k \geq 1$, any arc of the form $(a_{2k-1},a_{2k})$ is an \textbf{over arc}\index{arc!over} and any arc of the form $(a_{2k},a_{2k+1})$ is an \textbf{under arc}\index{arc!under}. The arc $(a_{r-1},a_r)$ is the \textbf{terminal arc}\index{arc!terminal}, and it may happen that the terminal arc is a base arc, or an over arc, or an under arc. If $(x,y)$ and $(z,w)$ are any two over arcs (potentially from different curves) then \[\min_P\{x,y\} <_P \min_P\{z,w\} <_P \max_P\{x,y\} <_P \max_P\{z,w\}, \] must not occur, as shown in Figure \ref{avoid}. Note that the $\min$ and $\max$ functions are taken with respect to the $<_P$-order.

\begin{figure}
\begin{center}
\includegraphics[scale=1]{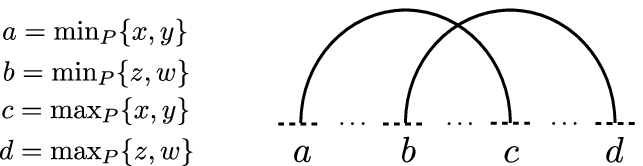}
\caption{The configuration of over arcs $(x,y)$, $(z,w)$ that may not occur in a virtual curve diagram.}\label{avoid}
\end{center}
\end{figure}

\end{itemize}
\end{definition}

\begin{remark}\label{conv}
We draw curve diagrams with a few conventions. The upper and base points lie on two horizontal lines oriented left to right. The upper line is called the \textbf{upper line}\index{line!upper} and contains the upper points and the lower line is called the \textbf{base line}\index{line!base} and contains the base points. The upper line is dashed. The $<_P$-order agrees with the order of the points along the oriented lines. The over and under arcs of $<_C$ are represented by arcs connecting the corresponding points. Base arcs will be represented by line segments connecting the base point to the corresponding upper point. Base points and terminal points will be emphasized with big dots. Note that in Definition \ref{vcd}, the inequality condition on over arcs is equivalent to the curves not intersecting above the upper line. Figure \ref{exconv} depicts the curve diagram of Example \ref{ex1}.
\end{remark}

\begin{example}\label{ex1}
Let $U$ be the set $\{a_1,b_1,b_2,b_3,c_1,c_2,c_3\}$ and $B$ the set $\{a_0,b_0,c_0\}$. Let $<_C$ be given by \[ a_0 <_C a_1 {\rm \ and \ } b_0 <_C b_1 <_C b_2 <_C b_3 {\rm \ and \ } c_0 <_C c_1 <_C c_2 <_C c_3,\] and let $<_P$ be given by $a_0<_P b_0<_P c_0$ and \[ c_3 <_P b_1 <_P a_1 <_P b_2 <_P c_1 <_P b_3 <_P c_2. \] The over arcs are $(b_1,b_2)$ and $(c_1,c_2)$ and one can see that the condition on over arcs is satisfied, so this is indeed a virtual curve diagram. This example is shown in Figure \ref{exconv}, following the conventions in Remark \ref{conv}.
\end{example}

\begin{example}
The trivial curve diagram $I_n$\index{virtual curve diagram!trivial $I_n$} is the vcd given by $B=\{b_1,b_2, \ldots, b_n\}$, $U=\{t_1,t_2,\ldots, t_n\}$, $b_i <_C t_i$ for $1 \leq i \leq n$, and \[ t_1 <_P t_2 <_P \ldots <_P t_n, \] and \[ b_1 <_P b_2 <_P \ldots <_P b_n. \]
\end{example}

\subsection{An equivalence relation}

We now define an equivalence relation $\sim$ on $\VCD_n$. This equivalence permits the moves $T$ and $B$ described in the introduction, as well as relabellings of the points of a curve diagram. If $x$ and $y$ are points, we say they are \textbf{adjacent} if there are no points strictly between them in the $<_P$ order.

\begin{definition}\label{equi}
Refer to Figure \ref{moves}. Let $D=(P,<_P,<_C)$ and $E=(Q,<_Q,<_R)$ be two vcds. Then $D$ and $E$ are related by a relabelling of their points if there exists a bijection $P \rightarrow Q$ that induces order isomorphisms $(P,<_P) \rightarrow (Q,<_Q)$ and $(P,<_C) \rightarrow (Q,<_R)$. 

We say $D$ and $E$ are related by a \textbf{$T$-move}\index{$T$-move}, if there exists a curve in $D$ containing a terminal arc $(a,b)$ where $a$ and $b$ are adjacent, and $Q=P \setminus \{b\}$, and the orders $<_Q$ and $<_R$ are the restrictions of $<_P$ and $<_C$ to $Q$. In this case, the point $a$ is a terminal point in $E$.

\begin{figure}
\begin{center}
\includegraphics[scale=.75]{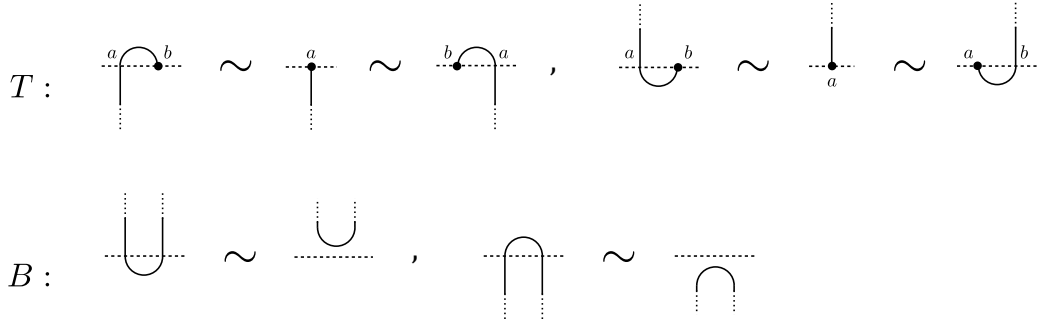}
\caption{The $T$ and $B$-moves.}\label{moves}
\end{center}
\end{figure}

We say $D$ and $E$ are related by a \textbf{$B$-move}\index{$B$-move}, if there exists a curve in $D$ containing three $<_C$-consecutive arcs $(a,b),(b,c),(c,d)$ where $b$ and $c$ are adjacent, and $Q=P \setminus \{b,c\}$, and the orders $<_Q$ and $<_R$ are the restrictions of $<_P$ and $<_C$ to $Q$.

If there exists a sequence $D_1, D_2, \ldots, D_n$ of vcds such that $D_k$ is related to $D_{k+1}$ by a relabelling or a $T$ or $B$-move for each $1 \leq k < n$, then we say $D_1$ is \textbf{equivalent}\index{equivalent vcds} to $D_n$, and write $D_1 \sim D_n$.
\end{definition}

\begin{remark}
Any diagram is equivalent to a unique diagram with a minimal amount of points. Such a diagram has no $T$ or $B$-moves possible which decrease the number of points. If a diagram is the minimal representative of its $\sim$-equivalence class, we will say that the diagram is \textbf{simplified}\index{simplified vcd}. Figure \ref{equiv} gives a pair of equivalent vcds, one of which is simplified.
\end{remark}

\begin{figure}
\begin{center}
\includegraphics[scale=.75]{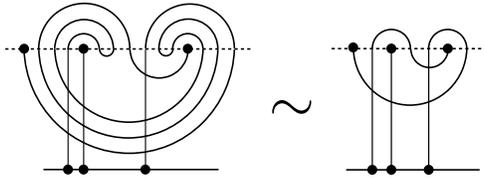}
\caption{A pair of equivalent vcds, one of which is simplified.}\label{equiv}
\end{center}
\end{figure}

\section{The action of virtual braids on virtual curve diagrams}\label{vaction}

In the introduction we claimed there are left and right actions of $\VB_n$ on $\VCD_n$. From now on we will only be interested in the left action. In this section we will clarify the action and prove some general facts.

\subsection{Defining the action and checking its well-definedness}

\begin{figure}
\begin{center}
\includegraphics[scale=.75]{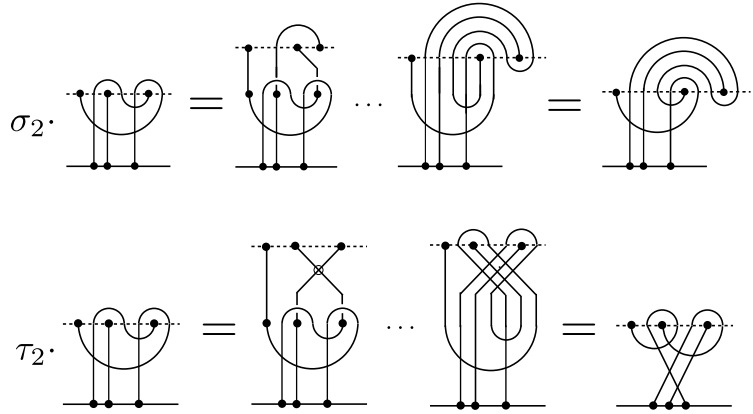}
\caption{Two examples of the left action.}\label{actexam}
\end{center}
\end{figure}

In the next section we will prove that the left action\index{left action on vcds} is faithful. To check that the action is well-defined, it is only necessary to check that $R_1 \cdot D = R_2 \cdot D$ for any $D \in \VCD_n$ and any defining relation $R_1=R_2$ in the presentation for $\VB_n$.

Figure \ref{welld} shows how one would prove this equality for the easiest relation $\sigma_1 \cdot (\sigma_1^{-1} \cdot D)=D$. We have simplified the situation in that the generic diagram $D$ could have many more over arcs than depicted in Figure \ref{welld}, however this fact won't change the overall proof. Note that the contents of the box labelled $D$ are irrelevant to the proof as well. In a similar way one would prove the relation $\tau_1 \cdot (\sigma_2 \cdot (\tau_1 \cdot D))=\tau_2 \cdot (\sigma_1 \cdot (\tau_2 \cdot D))$ as in Figure \ref{welld3}. Finally, Figure \ref{R3} shows the Reidemeister three relation $\sigma_1 \cdot (\sigma_2 \cdot (\sigma_1 \cdot D))=\sigma_2 \cdot (\sigma_1 \cdot (\sigma_2 \cdot D))$.

\begin{figure}
\begin{center}
\includegraphics[scale=.75]{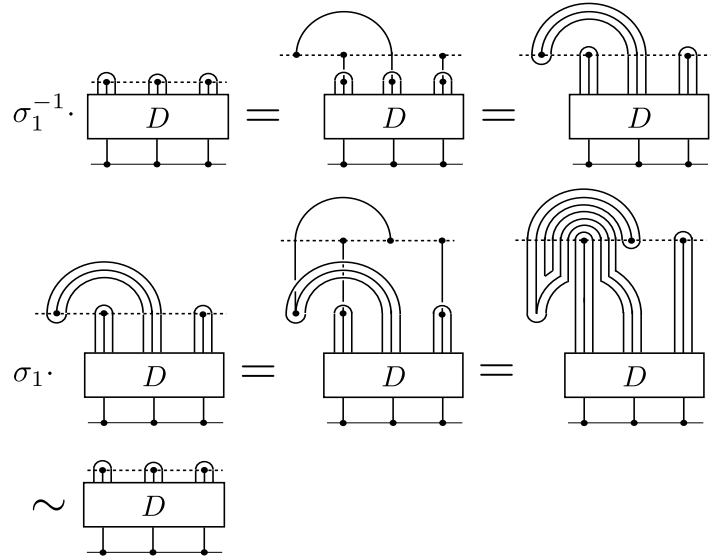}
\caption{The relation $\sigma_1 \cdot (\sigma_1^{-1} \cdot D) = D$ holds.}\label{welld}
\end{center}
\end{figure}

\begin{figure}
\begin{center}
\includegraphics[scale=.70]{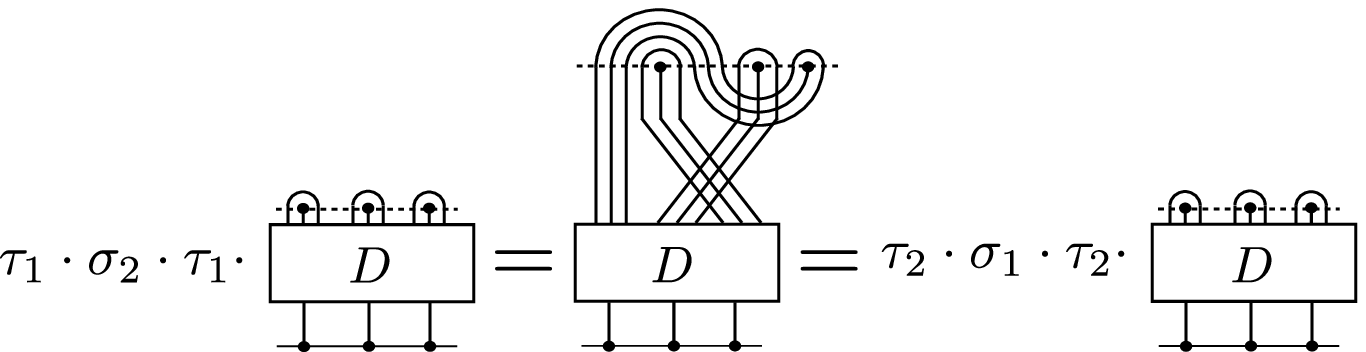}
\caption{The relation $\tau_1 \cdot (\sigma_2 \cdot (\tau_1 \cdot D))=\tau_2 \cdot (\sigma_1 \cdot (\tau_2 \cdot D))$ holds.}\label{welld3}
\end{center}
\end{figure}

\begin{figure}
\begin{center}
\includegraphics[scale=.70]{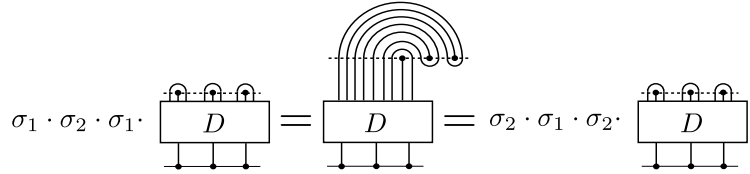}
\caption{The relation $\sigma_1 \cdot (\sigma_2 \cdot (\sigma_1 \cdot D))=\sigma_2 \cdot (\sigma_1 \cdot (\sigma_2 \cdot D))$ holds.}\label{R3}
\end{center}
\end{figure}

\subsection{Some properties of the action}

Let $D \in \VCD_n$ be a simplified virtual curve diagram. If $(a,b)$ is an over or under arc and $c$ is an upper point such that $a<_P c<_P b$ or $b<_P c<_P a$ then the arc $(a,b)$ \textbf{encloses}\index{enclose} the point $c$. If $(a,b)$ encloses both points of some arc $(c,d)$ then $(a,b)$ will be said to enclose $(c,d)$.

Assume that the terminal points of the curves in $D$ are \[t_1 <_P t_2 <_P \ldots <_P t_n.\] If a terminal over arc terminates at the terminal point $t_{i+1}$ and encloses $t_i$ we will say the terminal over arc is of \textbf{type} $(i,i+1)$\index{arc!type $(i,i+1)$,$(i+1,i)$} and if it terminates at $t_i$ and encloses $t_{i+1}$ we will say it is of type $(i+1,i)$. Any terminal over arc is of one of these types for some $i$. This is depicted in Figure \ref{toas}. Note that Figure \ref{exconv} in Section \ref{vcsec} depicts a vcd with no terminal over arcs. Note also that in any vcd with terminal over arcs as in Figure \ref{toas}, there certainly may be other arcs between $t_i$ and $t_{i+1}$, though they are not drawn in the figure.

\begin{figure}
\begin{center}
\includegraphics[scale=1.2]{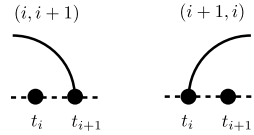}
\caption{The two types of terminal over arcs.}\label{toas}
\end{center}
\end{figure}

\begin{remark}\label{notoas}
Note that the curve diagram of a non-trivial braid $\beta \in \pB_n$ always has a terminal over arc. Indeed consider a simplified curve diagram that has no terminal over arcs. If it has only terminal base arcs it is the trivial diagram. If it has a terminal under arc, then consider an inner-most terminal under arc, i.e. a terminal under arc that encloses no terminal points. Then one sees that the diagram is not simplified as there will be a simplifying $T$ or $B$-move.
\end{remark}

Two under arcs $(a,b)$ and $(c,d)$ are said to \textbf{cross}\index{arc!crossing}, if $(a,b)$ encloses one of the points $c$ or $d$ but not both. Similarly an under arc $(a,b)$ and a base arc $(c,d)$ are said to cross if $(a,b)$ encloses $d$. Finally two base arcs $(a,b)$ and $(c,d)$ cross if $a<_P c$ and $d <_P b$. An under arc is said to be \textbf{free}\index{arc!free} if no other arcs cross it and it encloses no pair of crossing under arcs. 

\begin{proposition}\label{tbraid}
Given a simplified vcd $D$, there is a unique braid $\beta \in \pB_n$ such that the simplified representative of $\beta^{-1} \cdot D$ has no terminal over arcs.
\end{proposition}
\begin{proof}
Sitting on ``top" of any simplified vcd is the simplified curve diagram of a braid $\beta \in \pB_n$, along with some superfluous curves thrown in. This curve diagram contains all over arcs and free under arcs of $D$. Figure \ref{topbraid} gives an example of such a factorization. It follows that $\beta^{-1} \cdot D$ has no top braid and hence no terminal over arcs. Note that $\beta^{-1} \cdot D$ will still have non-terminal over arcs in general. We will denote $\beta$ by $t(D)$\index{top braid $t(D)$}.
\end{proof}

\begin{figure}
\begin{center}
\includegraphics[scale=.9]{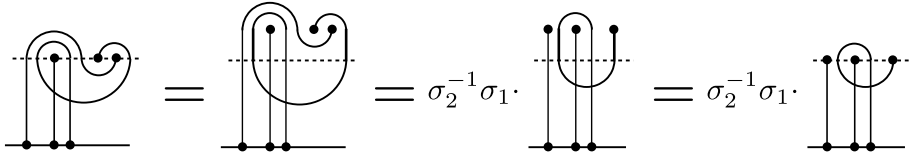}
\caption{A vcd with top braid $\sigma_2^{-1} \sigma_1$.}\label{topbraid}
\end{center}
\end{figure}

Based on the above proposition, an ordinary braid generator $\sigma_i$ acts on a vcd by acting on the curve diagram of the top braid, as well as any other superfluous curves in the way. For example, in the first row of Figure \ref{actexam}, the second puncture effectively travelled over and to the right of the third puncture, pulling any curves in the way along with it.

Now let's examine the action of permutations. We have already described how a transposition $\tau_i$ acts on a vcd. The following example describes the action of a more general permutation on a simplified vcd. 
 
\begin{example}\label{pexam}
Assume $\pi$ is a permutation and there exists some $1 \leq i,k$ with $i+k \leq n$ and $\pi(i+j)=\pi(i)+j$ for $1 \leq j \leq k$. That is, $\pi$ translates the interval $[i,i+k]$ to the interval $[\pi(i),\pi(i)+k]$ and is otherwise unrestricted. Let $D$ be a simplified vcd. Consider the region of $D$ above the upper line, and between the terminal points $t_i$ and $t_{i+k}$. This region will completely agree with the region above the upper line and between the terminal points $t_{\pi(i)}$ and $t_{\pi(i)+k}$ in the simplified representative of $\pi \cdot D$. This fact more or less completely determines how to figure out the action of a general permutation (aside from keeping track of the under arcs), since any permutation has a unique decomposition into such translations that are maximal. Consider the following example of the above description.

\begin{figure}
\begin{center}
\includegraphics[scale=1]{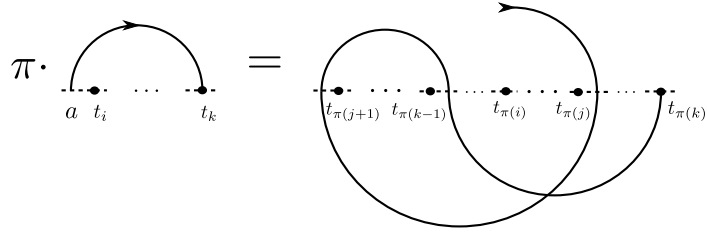}
\caption{A permutation acting on a terminal over arc $(a,t_k)$ with $a <_P t_k$.}\label{overarc2}
\end{center}
\end{figure}
\end{example} 

Suppose as in Figure \ref{overarc2}, a simplified vcd $D$ contains a terminal over arc $(a,t_k)$ (with $a<_P t_k$) enclosing the terminal points $t_i$ to $t_{k-1}$. Suppose for $i\leq j < k-1$ that $\pi$ maps the closed interval $[i,j]$ to $[\pi(i),\pi(j)]$ and $[j+1,k-1]$ to $[\pi(j+1),\pi(k-1)]$, whilst preserving the order and length of the intervals. Assume also that $\pi(k-1)<\pi(i)$ and $\pi(j)<\pi(k)$. Then the simplified representative of $\pi \cdot D$ will have an over arc enclosing at least the terminal points $t_{\pi(i)}$ to $t_{\pi(j)}$, followed by an under arc, followed by an over arc enclosing the terminal points $t_{\pi(j+1)}$ to $t_{\pi(k-1)}$, followed by an under arc terminating at $t_{\pi(k)}$ as shown in Figure \ref{overarc2}.

\section{Injectivity}\label{inj}

In this section we show that $\beta \mapsto \beta \cdot I_n$ is injective. We will define an equivalence relation $\leftrightarrow$ and a ``reduction" relation $\rightarrow$ on $\VCD_n$ and a complexity measure $c:\VCD_n \rightarrow \bZ_{\geq 0}$. We will then show that if $D \rightarrow E$ then $c(D)>c(E)$ and if $D \leftrightarrow E$ then $c(D)=c(E)$. We will then show that given any $D \in \VCD_n$ there is a sequence of reductions and $\leftrightarrow$-equivalences ending in a unique diagram of minimal complexity, up to $\leftrightarrow$-equivalence. Furthermore, the reductions and $\leftrightarrow$-equivalences will spell out a virtual braid word, and the corresponding virtual braid will be well-defined (i.e. will depend only on $D$).

In particular, for $D$ in the orbit $\VB_n \cdot I_n$, the unique minimal diagram will be $I_n$, up to $\leftrightarrow$-equivalence. The proof will be packaged in an application of the diamond lemma (see Lemma \ref{dlemma}).

Let's begin by introducing the measure of complexity of a simplified virtual curve diagram. For each $i$ let $o_i(D)$ equal the number of over arcs strictly enclosing $t_i$. Let the \textbf{complexity}\index{complexity $c(D)$} $c(D)$ equal $\sum_{i=1}^n o_i(D)$. If $D$ is not simplified, define $o_i(D)=o_i(E)$ and $c(D)=c(E)$ where $E$ is the simplified diagram equivalent to $D$.

\begin{definition}\label{reddef}
Define a binary relation $\rightarrow$ \index{$\rightarrow$-reduction} on $\VCD_n$ as follows:
\begin{enumerate}
\item If a simplified diagram $D$ contains a terminal over arc of type $(i,i+1)$ then let $D \rightarrow \sigma_i^{-1} \cdot D$.
\item If a simplified diagram $D$ contains a terminal over arc of type $(i+1,i)$ then let $D \rightarrow \sigma_i \cdot D$.
\end{enumerate}

\begin{remark}\label{compl}
If $D_1 \rightarrow D_2$ then $c(D_1)>c(D_2)$. If $D_1$ had a terminal over arc of type $(i,i+1)$ (resp. $(i+1,i)$) then multiplying by $\sigma_i^{-1}$ (resp. $\sigma_i$) would have the effect in Figure \ref{lesscomp}. Note that in the figure, the complexity of the right hand side is already less by $1$ than the complexity of the left hand side, and the right hand side might not even be simplified (simplifying would only potentially decrease the complexity). Note also that there may be arcs that are not drawn, but they are unchanged from the left hand side to the right hand side.

\begin{figure}
\begin{center}
\includegraphics[scale=1.3]{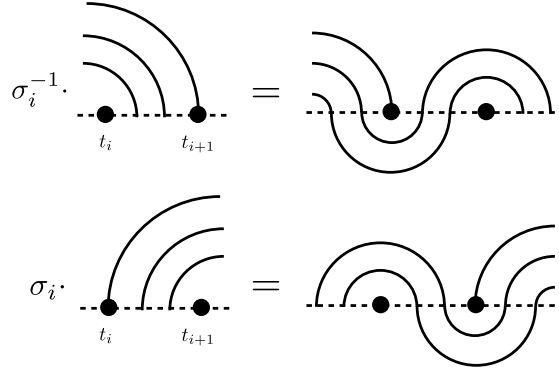}
\caption{The $\rightarrow$ relation decreases complexity.}\label{lesscomp}
\end{center}
\end{figure}

\end{remark}

\end{definition}

Define an equivalence relation $\leftrightarrow$\index{$\leftrightarrow$-equivalence} on virtual curve diagrams by $D \leftrightarrow E$ if there is a permutation $\pi$ such that $D=\pi \cdot E$. Denote the equivalence class of $D$ with square brackets, $[D]$. Similarly let $[\beta]$ be the left coset $S_n \beta \subset \VB_n$, where $S_n$ is the symmetric group generated by the $\tau_i$'s.

\begin{remark}\label{complperm}
If $D \leftrightarrow E$ then $c(D)=c(E)$. It is enough to check this in the case where $E=\tau_i \cdot D$. By the definition of the action of a generator $\tau_i$ (refer to Figure \ref{actexam}), one sees that $c(D) \geq c(\tau_i \cdot D)$, since the naive diagram for $\tau_i \cdot D$ might not be simplified. However by the same reasoning $c(\tau_i \cdot D) \geq c(\tau_i \cdot \tau_i \cdot D)=c(D)$ as well, so $c(D)=c(\tau_i \cdot D)$.
\end{remark}

Later we will need to apply the diamond lemma to the relation $\rightarrow$. We remind the reader now of the precise statement of the diamond lemma\index{diamond lemma}.

\begin{lemma}[Diamond lemma]\label{dlemma}
If $\rightarrow$ is a connected Noetherian binary relation (Noetherian: an infinite $a_1 \rightarrow a_2 \rightarrow \cdots$ is ultimately stationary), and if whenever $a \rightarrow b$ and $a \rightarrow c$ there is $d$ with $b \Rightarrow d$ and $c \Rightarrow d$ where $\Rightarrow$ is the reflexive transitive closure of $\rightarrow$, then there is a unique $m$ such that $\forall a, a \Rightarrow m$.
\end{lemma}

We will be using the diamond lemma on the relation $\rightarrow$ (defined on the quotient of $\VCD_n$ by the action of permutations, to be explained in Definition \ref{extend}). Thus the following proposition is important. It says that for any vcd $D$ and any braid generator $\sigma_i$ we have $D \rightarrow \sigma_i \cdot D$ or $\sigma_i \cdot D \rightarrow D$. Likewise for $\sigma_i^{-1}$ either $D \rightarrow \sigma_i^{-1} \cdot D$ or $\sigma_i^{-1} \cdot D \rightarrow D$. Thus any action of a braid generator either performs a reduction or reverses a reduction.

\begin{proposition}\label{obv}
Let $D$ be a simplified diagram and fix an $i$ with $1 \leq i \leq n-1$. If $D$ has no terminal over arc of type $(i+1,i)$ then $\sigma_i \cdot D$ has a terminal over arc of type $(i,i+1)$. Likewise, if $D$ has no terminal over arc of type $(i,i+1)$ then $\sigma_i^{-1} \cdot D$ has a terminal over arc of type $(i+1,i)$.
\end{proposition}

\begin{proof}
Since we are dealing with the action of braid generators on a vcd, this is really a statement about the curve diagram of the top braid $t(D)$ of $D$ (see Proposition \ref{tbraid}), and thus a statement about ordinary curve diagrams of ordinary braids.

Assume $D$ does not have a terminal over arc of type $(i+1,i)$. The possible configurations before and after multiplying by $\sigma_i$ are shown in Figure \ref{obvpic}. The right hand side of each equation has a terminal over arc of type $(i,i+1)$ as required.

However, the configuration in the bottom right has a special case which behaves slightly differently than the generic case in Figure \ref{obvpic} after multiplying by $\sigma_i$. It is shown in Figure \ref{special}. In the special case, the under arc terminating at $t_{i+1}$ encloses only the terminal point $t_i$ and a possibly empty collection of parallel non-crossing under arcs. After multiplying by $\sigma_i$ the diagram will be as in Figure \ref{special}, as opposed to Figure \ref{obvpic}. The right hand side still has a terminal over arc of type $(i,i+1)$, so we are done. The proof of the second half of the proposition, for $\sigma_i^{-1}$, follows by taking a mirror reflection of the above proof.

\begin{figure}
\begin{center}
\includegraphics[scale=1]{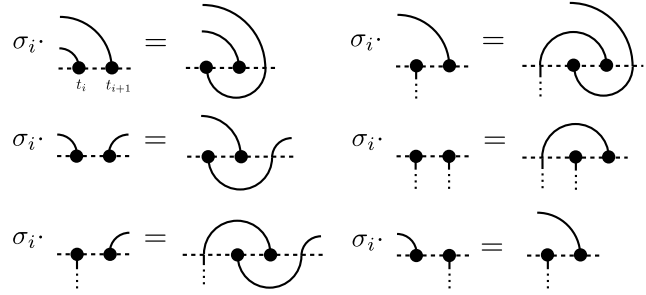}
\caption{Creating a terminal over arc of type $(i,i+1)$.}\label{obvpic}
\end{center}
\end{figure}

\begin{figure}
\begin{center}
\includegraphics[scale=1]{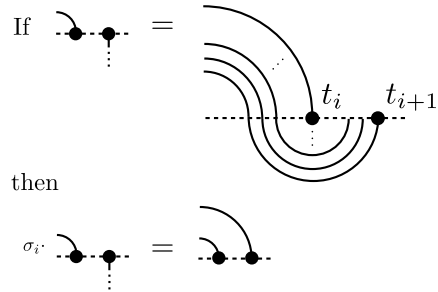}
\caption{A special case in Proposition \ref{obv}.}\label{special}
\end{center}
\end{figure}

\end{proof}

The following proposition will also be important when applying the diamond lemma to $\rightarrow$.

\begin{proposition}
Let $D$ be a simplified vcd and assume $D \rightarrow E$ and $D \rightarrow F$. Then there exists a vcd $G$ such that $E \Rightarrow G$ and $F \Rightarrow G$.
\end{proposition}
\begin{proof}
Consider the top braid $t(D)$. Since $D \rightarrow E$ and $D \rightarrow F$, we have $t(D) \rightarrow t(E)$ and $t(D) \rightarrow t(F)$ along the same terminal over arcs. By Remark \ref{notoas} we can reduce $t(E)$ until it is trivial and similarly reduce $t(F)$ until it is trivial. Thus $E \Rightarrow G$ and $F \Rightarrow G$ where $G=t(D)^{-1} \cdot D$.
\end{proof}

The following definition is related to Example \ref{pexam}. It considers when a permutation preserves a reduction, or equivalently when a permutation preserves the existence of a particular terminal over arc.

\begin{definition}\label{intact}
Let $D$ be a simplified diagram. Let $\pi$ be some permutation. Assume there is a terminal over arc of type $(i,i+1)$ or $(i+1,i)$ in $D$. We say it is left \textbf{intact}\index{intact} by $\pi$ or that the terminal over arc is intact in $\pi \cdot D$ if $\pi(i+1)=\pi(i)+1$ (note that it is the same condition for both types of terminal over arcs).
\end{definition}

\begin{remark}\label{intact2}
If $\pi$ leaves a terminal over arc of type $(i,i+1)$ (resp. $(i+1,i)$) intact, then the simplified representative of $\pi \cdot D$ will contain a terminal over arc of type $(\pi(i),\pi(i)+1)$ (resp. $(\pi(i)+1,\pi(i))$). If a terminal over arc is not left intact, then the simplified representative $\pi \cdot D$ will contain one of the two configurations in Figure \ref{notintact} (in the figure, we are depicting the case where the terminal over arc was of type $(i,i+1)$). In both configurations, the dotted gap might contain other terminal points. Note that if $\pi(i+1)=\pi(i)+1$ then $\pi \sigma_i^{\pm 1} = \sigma_{\pi(i)}^{\pm 1} \pi$. Note also that if a terminal over arc is left intact by a permutation, it is not the entire over arc that is left ``intact" (graphically speaking), but only at least the portion enclosing the terminal point of the arc and the adjacent terminal point.

\begin{figure}
\begin{center}
\includegraphics[scale=1.3]{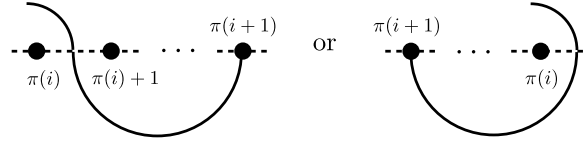}
\caption{The two configurations of a terminal over arc of type $(i,i+1)$ not left intact by a permutation.}\label{notintact}
\end{center}
\end{figure}

\end{remark}

The next proposition will tell us that one terminal over arc being intact does not interfere with another terminal over arc being intact. This will be convenient when proving the second part of Theorem \ref{diamond}.

\begin{proposition}\label{perms}
Let $D$ be a simplified diagram. Assume $D$ contains a terminal over arc of type $(i,i+1)$ and $\pi \cdot D$ contains a terminal over arc of type $(j,j+1)$ or $(j+1,j)$, where $\pi$ is some permutation. Then there is a permutation $\gamma$ such that both terminal over arcs are intact in $\gamma \cdot D$.
\end{proposition}
\begin{proof}
Note that our goal is to construct a permutation $\gamma$ such that $\gamma(i+1)=\gamma(i)+1$ and $\gamma \pi^{-1}(j+1)=\gamma \pi^{-1}(j)+1$. If $\pi$ leaves $(i,i+1)$ intact, then we can choose $\gamma=\pi$. Assume $\pi$ does not leave $(i,i+1)$ intact. Thus we have one of the configurations of Figure \ref{notintact}. Assume it is in the first configuration, i.e. $\pi(i+1)>\pi(i)+1$. There are two cases where it is not possible to construct a $\gamma$ (for combinatorial reasons). They are $\pi(i)=j$ and $\pi(i+1)=j+1$. We will show they cannot occur in a vcd.

Assume first that $\pi(i+1) > \pi(i)+1$ and $\pi(i)=j$. Then clearly $j+1=\pi(i)+1$ and the configuration present in $\pi \cdot D$ looks as in the left side of Figure \ref{proof13}. If we now apply $\pi^{-1}$ to arrive back at $D$, we will get the right side of Figure \ref{proof13} (this is due to the description in Example \ref{pexam}), which should have a terminal over arc of type $(i,i+1)$. However the presence of a terminal over arc of type $(j,j+1)$ in $\pi \cdot D$ prevents this.

Assume that $\pi(i+1) > \pi(i)+1$ and $\pi(i+1)=j+1$. The configuration must look as in the left side of Figure \ref{proof14}. After multiplying by $\pi^{-1}$ to get back to $D$, we have a contradiction. There cannot be a terminal over arc in $\pi \cdot D$ of type $(j+1,j)$ with $j+1=\pi(i+1)$ while at the same time a terminal over arc of type $(i,i+1)$ in $D$.

\begin{figure}
\begin{center}
\includegraphics[scale=1.3]{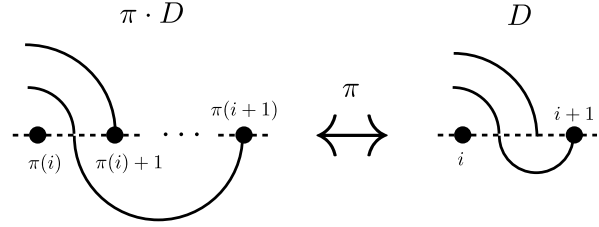}
\caption{Impossibility of the case $\pi(i)=j$ in Proposition \ref{perms}.}\label{proof13}
\end{center}
\end{figure}

\begin{figure}
\begin{center}
\includegraphics[scale=1.3]{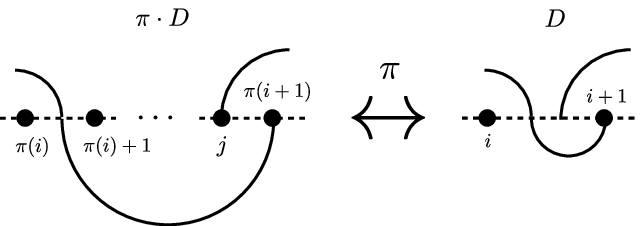}
\caption{Impossibility of the case $\pi(i+1)=j+1$ in Proposition \ref{perms}.}\label{proof14}
\end{center}
\end{figure}

\begin{figure}
\begin{center}
\includegraphics[scale=1.3]{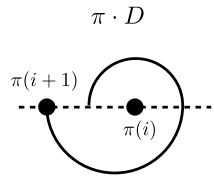}
\caption{Impossibility of the case $\pi(i+1)=\pi(i)-1=j$ in Proposition \ref{perms}.}\label{impcase}
\end{center}
\end{figure}

Assume now that the terminal over arc of type $(i,i+1)$ not left intact by $\pi$ is in the second configuration of Figure \ref{notintact}, that is, $\pi(i+1)<\pi(i)$. The cases where it is not possible to construct $\gamma$ are when $\pi(i)=j$, $\pi(i+1)=j+1$ and the special case $\pi(i+1)=\pi(i)-1=j$ or $\pi(i+1)=\pi(i)-1=j-1$. The first two cases are handled as before. Consider the special case $\pi(i+1)=\pi(i)-1=j$ or $\pi(i+1)=\pi(i)-1=j-1$. Then the terminal over arc of type $(i,i+1)$ must appear in $\pi \cdot D$ as in Figure \ref{impcase}. It is clear that there can be no terminal over arc terminating at $\pi(i)$.

In all cases aside from the above, a $\gamma$ satisfying the required conditions can be constructed combinatorially without issue.

\end{proof}

\begin{definition}\label{extend}
We can extend $\rightarrow$ to the set of pairs \[ \pP'= \{(\beta,D)| \beta \in \VB_n, D \in \VCD_n\}, \] given by $(\beta, D) \rightarrow (\sigma_i^{\pm 1} \beta, \sigma_i^{\pm 1} \cdot D)$ if $D \rightarrow \sigma_i^{\pm 1} \cdot D$ according to Definition \ref{reddef}.

Consider the set \[ \pP=\{ ([\beta], [D]) | \beta \in \VB_n, D \in \VCD_n\}. \] 

There is also an extension of $\rightarrow$ on $\pP$ given by $(\bd_1,\pD_1) \rightarrow (\bd_2,\pD_2)$ if there exists $\beta_1 \in \bd_1, D_1 \in \pD_1, \beta_2 \in \bd_2, D_2 \in \pD_2$ such that $(\beta_1,D_1) \rightarrow (\beta_2,D_2)$.

If there is a sequence $X \rightarrow Y \rightarrow \ldots \rightarrow Z$ we will write $X \Rightarrow Z$ (for $X,Y, \ldots, Z$ in $\VCD_n$ or $\pP'$ or $\pP$).
\end{definition}

The next theorem demonstrates that Lemma \ref{dlemma}, is applicable to $\rightarrow$ on $\pP$.

\begin{theorem}\label{diamond}
The relation $\rightarrow$ on $\pP$ satisfies the following properties:
\begin{enumerate}
\item There is no infinite sequence $(\bd_1,\pD_1) \rightarrow (\bd_2,\pD_2) \rightarrow (\bd_3,\pD_3) \rightarrow \ldots$.
\item If $(\bd,\pD) \rightarrow (\be,\pE)$ and $(\bd,\pD) \rightarrow (\bf,\pF)$ then there is a pair $(\bg,\pG)$ such that $(\be,\pE) \Rightarrow (\bg,\pG)$ and $(\bf,\pF) \Rightarrow (\bg,\pG)$.
\item The set $\{([\beta],[D]) \in \pP|\beta \cdot I_n = D\} \subset \pP$ is a connected component of $\rightarrow$.
\end{enumerate}
\end{theorem}
\begin{proof}
The first part follows from Remark \ref{compl}: if $D_1 \rightarrow D_2$ then $c(D_1)>c(D_2)$. 

For the second part let $(\bd,\pD) \rightarrow (\be,\pE)$ and $(\bd,\pD) \rightarrow (\bf,\pF)$. Assume that $\beta \in \bd$, $D \in \pD$, and $\pi$ is a permutation such that $D$ has a terminal over arc of type $(i,i+1)$ and $\pi \cdot D$ has a terminal over arc of type $(j,j+1)$ or $(j+1,j)$. Let $E=\sigma_i^{-1} \cdot D$ and $F=\sigma_j^{\pm 1} \cdot (\pi \cdot D)$.

By Proposition \ref{perms}, there is a permutation $\gamma$ such that the terminal over arc of type $(i,i+1)$ in $D$ and the terminal over arc of type $(j,j+1)$ or $(j+1,j)$ in $\pi \cdot D$ are intact in $\gamma \cdot D$. There are a few ways that two terminal over arcs can overlap, they are all depicted in Figure \ref{configs} up to a mirror reflection.

\begin{figure}
\begin{center}
\includegraphics[scale=1.3]{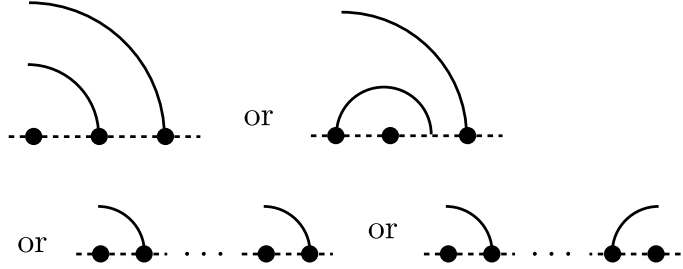}
\caption{The configurations of terminal over arcs in $\gamma \cdot D$.}\label{configs}
\end{center}
\end{figure}

\begin{figure}
\begin{center}
\includegraphics[scale=1.5]{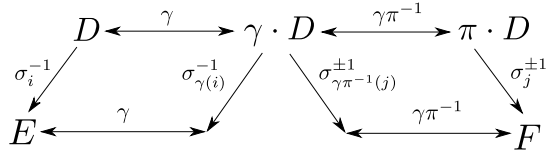}
\caption{A preliminary commutative diagram.}\label{comdiag}
\end{center}
\end{figure}

Since $\gamma$ leaves $(i,i+1)$ in $D$ intact in $\gamma \cdot D$ and $\gamma \pi^{-1}$ leaves $(j,j+1)$ or $(j+1,j)$ in $\pi \cdot D$ intact in $\gamma \cdot D$, by Remark \ref{intact2} we have \[ \gamma \sigma_i^{-1} = \sigma_{\gamma(i)}^{-1} \gamma, \] and \[ \gamma\pi^{-1} \sigma_j^{\pm 1} = \sigma_{\gamma\pi^{-1}(j)}^{\pm 1} \gamma\pi^{-1}.\] Thus we have the commutative diagram of Figure \ref{comdiag}.

In Figure \ref{comdiag} we have two reductions on $\gamma \cdot D$. We know by Proposition \ref{tbraid} that there is a common reduction $G=t(\gamma \cdot D)^{-1} \cdot (\gamma \cdot D)$. As well, by Artin's theorem, we know that both reduction paths from $\gamma \cdot D$ to $G$ spell out two braid words that are equivalent as braids (since curve diagrams faithfully represent ordinary braids). Thus we have a completed ``diamond" as in Figure \ref{comdiag2}.

\begin{figure}
\begin{center}
\includegraphics[scale=1.5]{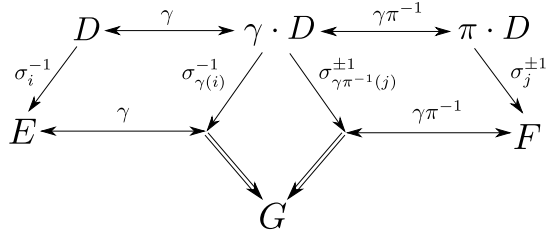}
\caption{A completed ``diamond".}\label{comdiag2}
\end{center}
\end{figure}

Finally we show that $\{([\beta],[D]) \in \pP|\beta \cdot I_n = D\}$ is a connected component of $\rightarrow$. By Proposition \ref{obv}, if $D$ does not have a terminal over arc of type $(i+1,i)$ then $\sigma_i \cdot D$ does so in this case $\sigma_i \cdot D \rightarrow D$. If $D$ does have a terminal over arc of type $(i+1,i)$ then $D \rightarrow \sigma_i \cdot D$. Thus for any $D$ either $D \rightarrow \sigma_i \cdot D$ or $\sigma_i \cdot D \rightarrow D$ (with a similar statement holding for $\sigma_i^{-1}$). 

Now any $D \in \VB_n \cdot I_n$ is of the form $\beta \cdot I_n$ for some virtual braid word $\beta$ so there is a sequence of diagrams \[ I_n=D_1,D_2,\ldots,D_r=D,\] such that for all $k \geq 1$ either $D_k \leftrightarrow D_{k+1}$ or $D_k \rightarrow D_{k+1}$ or $D_{k+1}\rightarrow D_k$. Thus $\{([\beta],[D]) \in \pP|\beta \cdot I_n = D\}$ is a connected component of $\rightarrow$.
\end{proof}

We can now prove injectivity.

\begin{theorem}\label{inject}
Let $\beta \in \VB_n$ such that $\beta \cdot I_n = I_n$. Then $\beta=1_{\VB_n}$.
\end{theorem}
\begin{proof}
By Lemma \ref{dlemma} and Theorem \ref{diamond}, there is a unique minimal element in $\pP$ with respect to $\rightarrow$. Clearly that element is $([1_{\VB_n}],[I_n])$.

If $\beta$ were not trivial, it certainly wouldn't be a permutation since $\pi \cdot I_n \neq I_n$ unless $\pi = 1_{\VB_n}$. Thus $([\beta],[I_n])$ would be a minimal element in $\pP$ distinct from $([1_{\VB_n}],[I_n])$, a contradiction.
\end{proof}

\begin{remark}
The above theorems show that each orbit of $\VB_n$ acting on $\VCD_n$ contains a unique (up to the action of permutations) element of minimal complexity, and that the action on that element is faithful.
\end{remark}

\section{The map $\psi$ and its kernel}\label{fut}
Consider the action induced by the map $\psi:\VB_n \rightarrow \Aut(F_{n+1})$ given in \cite{B,M}, where $F_{n+1}$ is the free group on generators $x_1,x_2,\ldots x_n,q$. The formula for $\psi$ is:

\[ \psi(\sigma_i)(x_j)=\begin{cases}x_j & j \neq i,i+1 \\ x_i x_{i+1}x_i^{-1} & j=i \\ x_i & j=i+1 \end{cases}, \] 
\[ \psi(\tau_i)(x_j)=\begin{cases}x_j & j \neq i,i+1 \\ qx_{i+1}q^{-1} & j=i \\ q^{-1}x_iq & j=i+1 \end{cases}, \] and $\psi(\sigma_i)(q)=\psi(\tau_i)(q)=q$.

\begin{remark}\label{vcdcalc}
Given a virtual braid $\beta$, there is a procedure, described to the author independently by V. G. Bardakov, D. Bar-Natan, and B. A. Cisneros De La Cruz, to calculate the free group elements $\psi(\beta)(x_i)$, directly from the vcd for $\beta$. We now describe this procedure and then use it to verify that a particular virtual braid is in $\ker \psi$ for $n=4$.

The free group element $\psi(\beta)(x_i)$ is of the form $U_ix_{\pi(i)}U_i^{-1}$. To calculate $U_i$ from the vcd of $\beta$, we begin at the base point of the $i^{\rm th}$ curve, and travel along this curve to its terminal point, building up $U_i$ along the way. If the $i^{\rm th}$ curve enters a crossing and the second curve crosses from the right, append a $q$ to the end of $U_i$. If the second curve crosses from the left, append a $q^{-1}$ to the end of $U_i$. If the $i^{\rm th}$ curve goes clockwise (or counterclockwise) over the $j^{\rm th}$ terminal point (in the $<_P$-order), then append $x_j$ (or $x_j^{-1}$ if counterclockwise) to $U_i$.
\end{remark}

\begin{example}
One can determine, using Theorem \ref{diamond}, that the vcd in Figure \ref{exconv} comes from the braid $\tau_1 \tau_2 \sigma_2 \tau_2 \sigma_1 \tau_1$. Following the above recipe, the action of this braid on $(x_1,x_2,x_3)$ is \[(q^2x_2q^{-2},x_2qx_3q^{-1}x_2^{-1},x_3q^{-3}x_1q^3x_3^{-1}),\] which can be verified by direct calculation via the formula for $\psi$ as well.
\end{example}

\begin{figure}
\begin{center}
\includegraphics[scale=1.5]{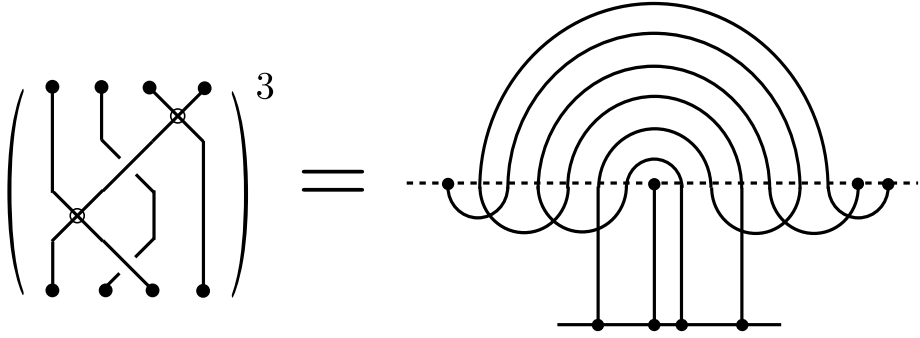}
\caption{The element $(\tau_3 \sigma_2 \tau_1 \sigma_2^{-1})^3$ of $\ker (\psi:\VB_4 \rightarrow \Aut(F_5))$.}\label{kern}
\end{center}
\end{figure}

With some trial and error, one may find a particular element of the kernel of $\psi$\index{kernel of $\psi$}:

\begin{proposition}
The virtual braid $\beta=(\tau_3 \sigma_2 \tau_1 \sigma_2^{-1})^3 \in \VB_4$ is in $\ker (\psi:\VB_4 \rightarrow \Aut(F_5))$.
\end{proposition}
\begin{proof}
Figure \ref{kern} depicts $\beta$ and its non-trivial vcd. Note that the braid word read left to right translates to the braid diagram read top to bottom. Using the procedure in Remark \ref{vcdcalc}, one can quickly verify that the action of $\psi(\beta)$ is indeed trivial. Alternatively, this can and should be (and was) checked with the aid of a computer.
\end{proof}

As far as the author can tell, it remains possible that $\psi$ is injective for $n=3$. The set of finite type invariants of pure $3$-strand virtual braids is known to form a complete invariant, as shown in \cite{BMVW}. This would also follow from the injectivity of $\psi$ for $n=3$.

\end{document}